\documentclass{amsart}

\usepackage{graphicx}
\usepackage{multicol}
\usepackage{float}
\usepackage{changepage}

\newtheorem{theorem}{Theorem}[section]
\newtheorem{lemma}[theorem]{Lemma}
\newtheorem{prop}[theorem]{Proposition}
\newtheorem{cor}[theorem]{Corollary} 

\theoremstyle{definition}
\newtheorem{definition}[theorem]{Definition}

\theoremstyle{remark}

\theoremstyle{question}
\newtheorem{question}[theorem]{Question}

\numberwithin{equation}{section}

\newcommand{\h}{\mathcal{H}}
\newcommand{\D}{\mathbb{D}}
\newcommand{\C}{\mathbb{C}}

\newcommand{\R}{\mathbb{R}}
\newcommand{\T}{\mathbb{T}}

\begin{document}

\title{Convexity of the Berezin Range}

\author{Carl C. Cowen}
\author{Christopher Felder}

\address{Department of Mathematical Sciences, Indiana University -- Purdue University, Indianapolis, Indianapolis, Indiana 46202-3216}
\email{ccowen@iupui.edu}

\address{Department of Mathematics and Statistics, Washington University In St. Louis, St. Louis, Missouri, 63136}
\email{cfelder@wustl.edu}

\subjclass[2010]{Primary 47B32, Secondary 52A10}
\date{\today}


\keywords{Berezin transform, Berezin range, Berezin set, convexity, composition operator}
\thanks{The second author was supported in part by NSF Grant \#1565243.}

\begin{abstract}
This paper discusses the convexity of the range of the Berezin transform. For a bounded operator $T$ acting on a reproducing kernel Hilbert space $\h$ (on a set $X$), this is the set $B(T) : = \{ \langle T\hat{k}_x, \hat{k}_x \rangle_\h : x \in X \}$, where $\hat{k}_x$ is the normalized reproducing kernel for $\h$ at $x \in X$. Primarily, we focus on characterizing convexity of this range for a class of composition operators acting on the Hardy space of the unit disk. 
\end{abstract}
\maketitle

\section{Introduction and Notation}\label{intro}
Throughout this paper we work in $\textit{reproducing kernel Hilbert space}$ (RKHS). These spaces are complete inner-product spaces comprised of complex-valued functions defined on a set $X$, where point evaluation is bounded. Formally, that is, if $X$ is a set and $\h$ is a subset of all functions $X \to \C$, then $\h$ is an RKHS on $X$ if it is a complete inner product space and point evaluation at each $x \in X$ is a bounded linear functional on $\h$. Via the Riesz representation theorem, we the know if $\h$ is an RKHS on $X$, there is a unique element $k_x \in \h$ such that $\langle f, k_x \rangle_{\h} = f(x)$ for each $x \in X$ and all $f \in \h$. The element $k_x$ is called the $\textit{reproducing kernel at x}$. Further, we will denote the normalized reproducing kernel at $x$ as $\hat{k}_x:= k_x/ \|k_x \|_\h$.
For more on the general theory of this topic, we point the reader to \cite{MR1882259} and \cite{MR3526117}. 
From here out, $\h$ is assumed to be an RKHS (on some set $X$).

For a bounded linear operator $T$ acting on $\h$, the \textit{Berezin range of $T$}
is 
\[
B(T) := \left\{ \langle T\hat{k}_x, \hat{k}_x \rangle  : x \in X \right\}. 
\]
The convexity of the Berezin range is the focus of this paper. After some discussion, we give characterizations for certain concrete operators to have convex Berezin range. Namely, we characterize convex Berezin range for matrices and multiplication operators in Section \ref{matrix_mult}, and composition operators with automorphic symbol (acting on $H^2$) in Section \ref{comp_ops}. We conclude with some open questions in Section \ref{questions}.


A well known example of an RKHS, and the main space in which we work here, is the classical Hilbert Hardy space on the unit disk $\D$,
\[
H^2 := \left\{ f(z) = \sum_{n\ge 0}a_n z^n \in \textrm{Hol}(\D) : \sum_{n\ge 0} |a_n|^2 < \infty \right\},
\]
where $\textrm{Hol}(\D)$ is the collection of holomorphic functions on $\D$. 
If $f(z) = \sum_{n\ge 0}a_nz^n$ and $g(z) = \sum_{n\ge 0}b_nz^n$ are elements of $H^2$, their inner product is given by 
\[
\langle f, g \rangle_{H^2} = \sum_{n \ge0}a_n \overline{b_n},
\]
with, of course, $\|f\|_{H^2}^2 = \sum_{n \ge 0} |a_n|^2$. 
One may check that the reproducing kernel for $H^2$, known as the Szeg\H{o} kernel, is given by
\[
k_w(z) = \frac{1}{1 - \overline{w}z} \ , \ \ z,w \in \D.
\] 
We choose to focus primarily on this space because it is well studied and its kernel is quite simple. However, another RKHS to keep in mind is the Bergman space on a complex domain, say, $\Omega \subset \C^n$, 
\[
A^2(\Omega) := \left\{ f \in \operatorname{Hol}(\Omega) : \int_\Omega |f(z)|^2 dV(z) < \infty \right\},
\]
where $dV$ is normalized volume measure on $\Omega$. 

Before moving on, let us collect some definitions. 

\begin{definition}
Let $\h$ be an RKHS on a set $X$ and let $T$ be a bounded linear operator on $\h$. 
\begin{enumerate}
\item For $x\in X$, the $\textit{Berezin transform}$ of $T$ at $x$ (or \textit{Berezin symbol of $T$}) is
\[
\tilde{T}(x) := \langle T\hat{k}_x, \hat{k}_x \rangle_\h.
\]
\item The \textit{Berezin range} of $T$ (or \textit{Berezin set of $T$}) is 
\[
B(T) := \left\{\langle T\hat{k}_x, \hat{k}_x \rangle_\h : x \in X \right\}.
\]
\item The \textit{Berezin radius of $T$} (or \textit{Berezin number of $T$}) is
\[
b(T) : = \sup_{x \in X} |\tilde{T}(x)|.
\]
\end{enumerate}
\end{definition}

We will now give some expository remarks to motivate our results, and to provide a highlight reel of the uses of the Berezin transform (the likes of which seem to be absent from existing literature). 

\section{Background and Motivation}\label{background}

The Berezin set and number, also denoted by $\operatorname{Ber}(T)$ and $\operatorname{ber}(T)$, respectively, were purportedly first formally introduced by Karaev in \cite{MR2253012}.
The Berezin transform itself was introduced by F. Berezin in \cite{MR0350504} and has proven to be a critical tool in operator theory, as many foundational properties of important operators are encoded in their Berezin transforms.

One of the first important results involving the Berezin transform involves the invertibility of Toeplitz operators acting on $H^2$. 
In \cite{MR0361894}, R.G. Douglas asked the following: if $\varphi \in L^{\infty}(\T)$ with $|\tilde{T_\varphi}| \ge \delta > 0$, is the Toeplitz operator $T_\varphi$ invertible? Tolokonnikov \cite{MR629839}, and then Wolff \cite{MR1979771}, showed that when $\delta$ is sufficiently close to 1, the answer is affirmative. The lower bound on $\delta$ was then improved by Nikolskii \cite{MR827223}. 
Much later, Karaev proved similar results for certain general operators acting on RKHS's:  if the modulus of the Berezin transform of a suitably nice operator is sufficiently bounded away from zero, then the invertibility of the operator is ensured \cite[Theorem~3.4]{MR2253012}. Similar recent results for closed range type properties of Toeplitz operators can be found in \cite{MR4025066}. 

Following Douglas' question, Berger and Coburn \cite{MR843308} asked something similar: if the Berezin symbol of an operator on the Hardy or Bergman space vanishes on the boundary of the disk, must the operator be compact? This question was addressed by Nordgren and Rosenthal \cite{MR1320554}, where they presented several counterexamples. 
However, Nordgren and Rosenthal showed, on a so-called \textit{standard} RKHS, that if the Berezin symbols of all unitary equivalents of an operator vanish on the boundary, then the operator is compact. The counterexamples presented come in the form of composition operators, and motivate the study in Section \ref{comp_ops}.

Another important theorem here, due to Axler and Zheng \cite{MR1647896},
is that if $T$ is a finite sum of finite products of Toeplitz operators acting on the Bergman space of the unit disk, then $T$ is compact if and only if the Berezin transform of $T$ vanishes as it approaches the boundary of the disk. Shortly after, Engli\v{s} \cite{MR1682815} generalized this result to weighted Bergman spaces on bounded symmetric domains in several variables. Later, Su\'{a}rez \cite{MR2360608} proved an analogous result for \textit{any operator in the Toeplitz algebra} of the Bergman space on the unit ball. 
Some other Axler-Zheng type results for various spaces over several types of domains can be found in \cite{MR1937440, MR3008923, MR3743367, MR2943731, MR3116666, MR4078705}.

Some other quite interesting results involving applications of the Berezin transform include the  characterization of invertible operators which are unitary \cite{MR3088319}, characterizations of Schatten-von Neumann class membership \cite{MR3100410, MR1942546, MR2926635}, Beurling-Arveson type theorems for some RKHS's \cite{MR2442904}, a characterization of skew-symmetric operators \cite{MR4203306}, and the characterization of truncated Toeplitz operators with bounded symbols, along with descriptions of invariant subspaces of isometric composition operators \cite{MR3522132}. See also \cite{MR3079840}, which will be discussed in Section \ref{questions}. A detailed introduction to the Berezin transform of operators on spaces of analytic functions can be found in \cite{MR1457018}.

We note that the range of the Berezin transform has been studied from a function theoretic viewpoint, for example in work by Ahern, which establishes a Brown-Halmos type theorem for the Bergman space \cite{MR2085115} (see also \cite{MR3747963, MR2886615}). However, apart from some examples due to Karaev \cite[Section~2.1]{MR3079840}, it does not appear that the Berezin range has been studied from a set-theoretic or geometric viewpoint, as we will discuss here.

\subsection{Numerical Range and the Toeplitz-Hausdorff Theorem}

In an RKHS, the Berezin range of an operator $T$ is a subset of the numerical range of $T$,
\[
W(T) : = \left\{ \langle Tu, u \rangle : \|u \| = 1 \right\}.
\]
The numerical range of an operator has some interesting properties. For example, it is well known that the spectrum of an operator is contained in the closure of its numerical range. Further, the numerical range of an operator is always \textit{convex}-- this results is known as the Toeplitz-Hausdorff Theorem \cite{MR1544315, MR1544350}. For more background on the numerical range, we point the reader to \cite{MR1417493}. Much effort has gone into describing the geometry of the numerical range (e.g. see \cite{MR2378310, MR3932079}), but to the knowledge of the authors, there are only a handful of results describing the geometry of the Berezin range \cite[Section~2.1]{MR3079840}, none of which address convexity.

As the convexity of the numerical range is arguably its most enigmatic property, we are motivated to ask the main question addressed in this paper: 
\begin{adjustwidth}{.5cm}{.5cm}
\textit{Given a bounded operator $T$ acting on an RKHS $\h$, is $B(T)$ convex?  Conversely, if $B(T)$ is convex, what can be said of $T$?}
\end{adjustwidth}
This question was initially pointed out by Karaev \cite{MR3079840}, and we give answers for a few classes of concrete operators. In general, as we will see, the Berezin range of an operator is $\textit{not}$ convex. 
\begin{figure}[H]
\includegraphics[width=.82\textwidth]{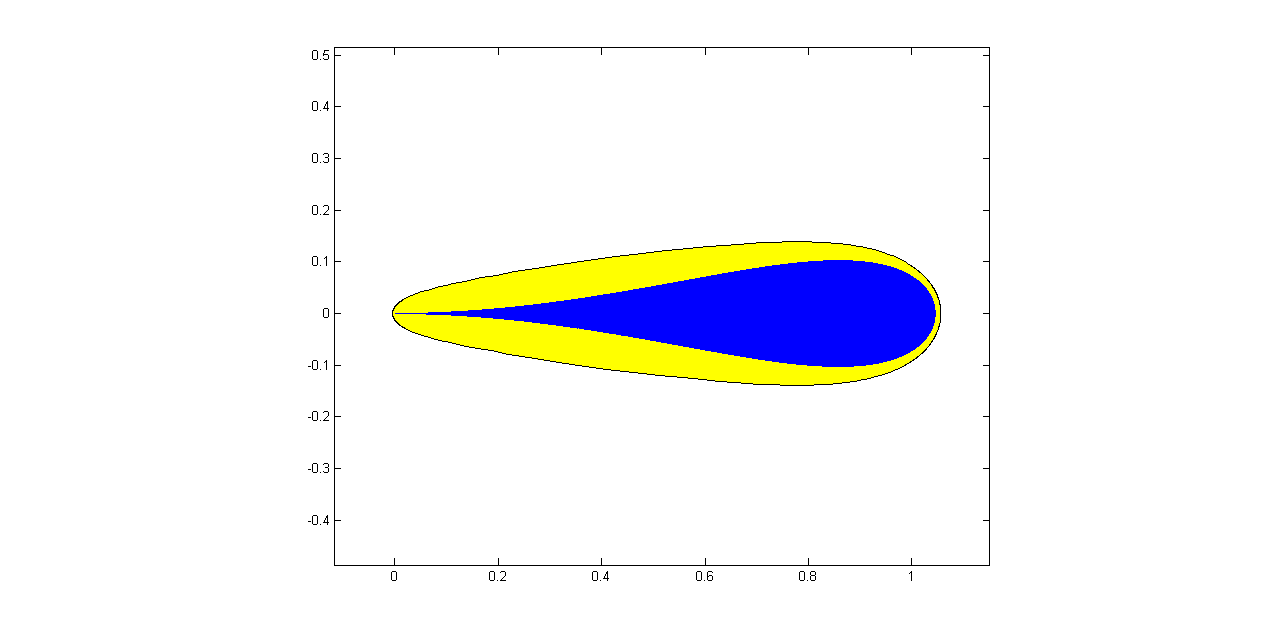}
\caption{The numerical and Berezin ranges of the composition operator with symbol $\frac14 (1 + z)^2$ acting on $H^2$, with the Berezin range appearing to be non-convex.}
\end{figure}
What fails in the Toeplitz-Hausdorff Theorem when restricting to normalized reproducing kernels? 
The proof of the theorem relies on the following result: \\[.2cm]
\noindent\textbf{Elliptical Range Theorem.}
\textit{Let $A$ be a $2\times2$ matrix with complex entries and eigenvalues $\lambda_1$ and $\lambda_2$. Then $W(A)$ is an elliptical disk with $\lambda_1$ and $\lambda_2$ as foci, and $\{\text{tr}(A^*A) - |\lambda_1|^2 - |\lambda_2|^2\}^{1/2}$ as its minor axis.} \\[.2cm]
There are several proofs of this theorem (e.g. \cite{MR96127, MR1322932}) but a common thread is to produce a linear combination of unit vectors, that is, again, a unit vector. In general, this is simply impossible to do for normalized reproducing kernels; for any two points $x_1$ and $x_2$ and constants $c_1$ and $c_2$, 
\[
\frac{c_1\hat{k}_{x_1} + c_2\hat{k}_{x_2}}{\|c_1\hat{k}_{x_1} + c_2\hat{k}_{x_2}\|}
\]
is not a normalized reproducing kernel. 
Another simple proof of the Toeplitz-Hausdorff theorem using an adaptation of the idea above can be found in \cite{MR262849}.


\section{Finite Dimensions and Multiplication Operators}\label{matrix_mult}
In this section, we will give characterizations for convexity of the Berezin transform of some easily understood classes of bounded linear transformations. 
\subsection{The Berezin range of matrices}
As a primer, let us consider the finite dimensional setting. Let $v = (v_1, \ldots, v_n) \in \C^n$ and $X = \{ 1, \ldots, n \}$. We can consider $\C^n$ as the set of all functions mapping $X \to \C$ by $v(j) = v_j$. Letting $e_j$ be the $j$th standard basis vector for $\C^n$ under the standard inner product, we can view $\C^n$ as an RKHS with kernel
\[
k(i,j) = \langle e_j, e_i \rangle.
\]
Note that $k_j = \hat{k}_j$ for each $j=1, \ldots, n$. For any complex $n\times n$ matrix $A = (a_{jk})_{j,k = 1}^n$, we have $\langle Ae_j, e_j \rangle = a_{jj}$. Thus, the Berezin range of $A$ is simply
\[
B(A) = \{ a_{jj} : j = 1, \ldots, n \},
\]
which is just the collection of diagonal elements of $A$. It is immediate that the only way this set can be convex is if the diagonal elements of $A$ are all equal: 
\begin{prop}
Let $A$ be an $n\times n$ matrix with complex entries. Under the standard inner product for $\C^n$, the Berezin range of $A$ is convex if and only if $A$ has constant diagonal. 
\end{prop}
This proposition shows that the geometry of the Berezin range of a matrix is remarkably simple compared to the numerical range of the matrix. We also point out that the trace of a matrix can be recovered as the sum over the elements of its Berezin range. More generally, it is known in some spaces (e.g. see \cite[Proposition~3.3]{MR2311536}) that when $T$ is a trace-class (or positive) operator, the trace of $T$ can be recovered using the Berezin transform.
We also point to the reader to \cite[Chapter~3]{MR2311536} for connections of the Berezin transform with the Fock space and BMO. 

As the Berezin transform is immediately understood in the finite dimensional setting, we quickly shift to infinite dimensional spaces.


\subsection{Multiplication Operators}\label{mult_ops}
Recall that for any Hilbert (or Banach) space of functions $\h$, 
\[
\operatorname{Mult}(\h) : = \{ g\in \h : gf \in \h \ \text{for all } f\in \h \}.
\]
For $g \in \operatorname{Mult}(\h)$, define the multiplication operator $M_g$ on $\h$ by $M_g f = gf$. 
\begin{prop}\label{mult_op_con}
Let $\h$ be an RKHS on a set $X$ and $g \in \text{Mult}(\h)$. Then the Berezin range of $M_g$ is convex if and only if $g(X)$ is convex. 
\end{prop}
\begin{proof}
Let $x \in X$ and observe that
\[
\tilde{M}_g(x) 
= \langle  g\hat{k}_x,  \hat{k}_x \rangle_{\h}
= \frac{1}{\| k_x\|_{\h}^2} \  \langle  g k_x,   k_x \rangle_{\h}
= \frac{1}{\| k_x\|_{\h}^2} \   g(x) k_x(x)
= g(x).
\]
Thus, $B(M_g) = g(X)$, and the result follows. 
\end{proof}

Similar to the matricial case, this characterization of convexity is exceptionally simple. 
One may think that convexity of the Berezin range could be simply understood. However, this is not generally the case. 
In order to demonstrate this, we move to some classes of operators acting on $H^2$ where the characterization of convexity becomes more technically involved.


\section{Composition Operators on $H^2$}\label{comp_ops}
A composition operator $C_\varphi$, induced by a complex-valued function $\varphi : X \to X$ (known as the \textit{symbol} of the operator, not to be confused with the Berezin symbol of the operator), acts on a space of functions defined on $X$ by 
\[
C_\varphi f := f \circ \varphi.
\]
These operators are beloved by many and have a long and important history in function and operator theory (e.g. see the monographs \cite{MR1397026, MR1237406}). 
One motivation for studying the Berezin range of these operators is that they often elude Axler-Zheng type results; e.g. there are composition operators such that 
$\tilde{C}_\varphi(z) \to 0$ as $z \to \partial X$, but $C_\varphi$ is \textit{not compact} (see \cite[Theorem~2.3]{MR1320554}).

\subsection{Eliptic Symbols}
We begin by considering a very elementary class of composition operators acting on $H^2$. For $\zeta \in \T$ (the complex unit circle), consider the elliptic automorphism of the disk
\[
\varphi(z)= \zeta z.
\]
Acting on $H^2$, these operators have Berezin transform
\begin{align*}
\tilde{C}_\varphi(z) &= \langle C_{\varphi} \hat{k}_z, \hat{k}_z \rangle \\
&= \left(1 - |z|^2\right)\langle C_{\varphi} k_z, k_z \rangle\\
&= \frac{1 - |z|^2}{1 - |z|^2\zeta}.
\end{align*}

With a little work, we come to a characterization of the convexity of $B(C_\varphi)$.
\begin{figure}[H]
\includegraphics[scale=0.3]{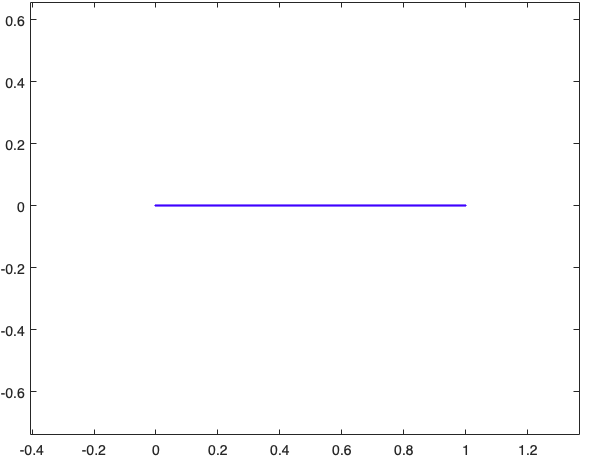}
\includegraphics[scale=0.3]{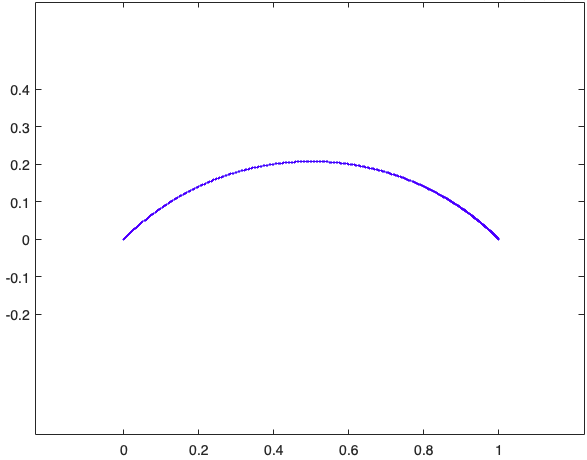}
\caption{$B(C_\varphi)$ on $H^2$ for $\zeta = -1$ (left, apparently convex) and $\zeta = i\pi/4$ (right, apparently not convex).}
\end{figure}
\begin{theorem}\label{elliptic}
Let $\zeta \in \T$ and $\varphi(z) = \zeta z$. Then the Berezin range of $C_\varphi$ acting on $H^2$ is convex if and only if $\zeta = 1$ or $\zeta = -1$. 
\end{theorem}
\begin{proof}
Let us prove the backward implication first. Suppose first that $\zeta = 1$ so $\varphi(z) = z$. Putting $z= re^{i\theta}$ with $0 \le r < 1$, the calculations above show that
\[
\tilde{C}_\varphi(re^{i\theta}) = \frac{1 - r^2}{1 - r^2} = 1 
\]
so $B(C_\varphi) = \{1\}$, which is convex. 
Similarly, for $\varphi(z) = -z$, we obtain
\[
B(C_\varphi) =\left\{ \frac{1 - r^2}{1 + r^2} : r \in [0,1) \right\} = (0,1],
\]
which is also convex. 

Conversely, suppose that $B(C_\varphi)$. We have that 
\[
\tilde{C}_\varphi (re^{i\theta}) =  \frac{1 - r^2}{1 + r^2\zeta},
\]
which is a function independent of $\theta$. Hence, by definition, $B(C_\varphi)$ is just a path in $\C$. By convexity, $B(C_\varphi)$ must then be either a point or a line segment. It is immediate that $B(C_\varphi)$ is a point if and only if $\zeta = 1$, so let us assume $B(C_\varphi)$ is a line segment. Note that $\tilde{C}_\varphi(0) = 1$ and that $\lim_{r\to 1^-}\tilde{C}_\varphi (re^{i\theta}) = 0$. This tells us that $B(C_\varphi)$ must be a line segment passing through the point 1 and approaching the origin. Consequently, we must have $\Im \left\{B(C_\varphi)\right\} = \{0\}$, which can happen if and only if $\Im \{ \zeta \} = 0$. Thus, as $\zeta \in \T$, we have $\zeta = 1$ or $\zeta = -1$, the former of which was handled. 
\end{proof}
This theorem characterizes the convexity of the Berezin transform for composition operators with the elliptic automorphisms, which belong to a wider class of composition operators with automorphic symbols. We turn to characterize the convexity of the Berezin range of another such class.

\subsection{Composition Operators with Blaschke Factor Symbol}
For $\alpha \in \D$, consider the automorphism of the unit disk (known as a Blaschke factor)
\[
\varphi_{\alpha}(z):=\frac{z - \alpha}{1-\overline{\alpha}z}
\]
and the composition operator
\[
C_{\varphi_\alpha}f = f \circ \varphi_\alpha. 
\]
Acting on $H^2$, we have
\begin{align*}
\tilde{C}_{\varphi_\alpha}(z) &= \langle C_{\varphi_\alpha} \hat{k}_z, \hat{k}_z \rangle \\
&= \left(1 - |z|^2\right)\langle C_{\varphi_\alpha} k_z, k_z \rangle\\
& = \left(1 - |z|^2\right) k_z(\varphi_\alpha(z))\\
&= \frac{1 - |z|^2}{1 - \overline{z}\varphi_\alpha(z)}.
\end{align*}

With the aid of a computer, we can plot an example: 
\begin{figure}[H]
\includegraphics[scale=0.5]{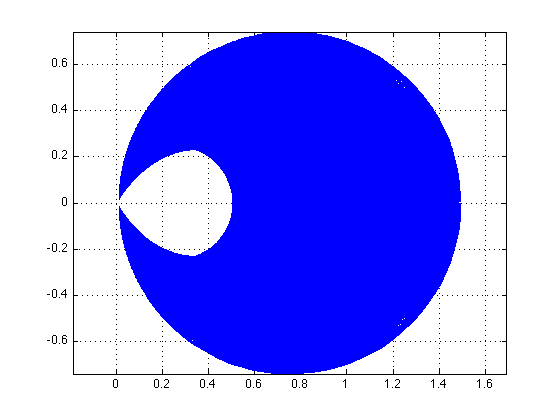}
\caption{$B(C_{\varphi_\alpha})$ on $H^2$ for $\alpha = -1/2$}
\end{figure}

In the above example, the Berezin range is more geometrically interesting than in the elliptic case, and is clearly not convex(!).
We will ultimately give a characterization for convexity of the Berezin range in this case, but require some calculations and observations first.

\begin{lemma}\label{real_imag_parts}
On $H^2$, the real and imaginary parts of $\tilde{C}_{\varphi_\alpha}$ are given by
\[
\Re\left\{\tilde{C}_{\varphi_\alpha}(z)\right\} = c_{\alpha,z} \left(1 - |z|^2\right)\left(1-\Re\{\overline{\alpha} z\}) + 2 (\Im\{\overline{\alpha}z\})^2 \right)
\]
and
\[
\Im\left\{\tilde{C}_{\varphi_\alpha}(z)\right\} =  c_{\alpha,z} \Im\{\overline{\alpha}z\} \left(1+|z|^2 - 2\Re\{\overline{\alpha} z\} \right),
\]
where
\[
c_{\alpha,z} = \frac{1 - |z|^2}{\left|1 - |z|^2 + 2i\Im\{\alpha\overline{z}\}\right|^2}.
\]
\end{lemma}
\begin{proof}
Let us make some computations: 
\begin{align*} 
\tilde{C}_{\varphi_\alpha}(z) &= \frac{1 - |z|^2}{1 - \overline{z}\varphi_\alpha(z)}\\
&= \frac{(1 - |z|^2)(1 - \overline{\alpha} z)}{1 - \overline{\alpha} z - \overline{z}(z - \alpha)}\\
&= \frac{(1 - |z|^2)(1 - \overline{\alpha} z)}{1 - |z|^2 + 2i\Im\{\alpha\overline{z}\}}.
\end{align*}
Multiplying by a complex conjugate in the denominator, we have 
\begin{align*}
\tilde{C}_{\varphi_\alpha}(z) &= c_{\alpha,z} \left(1 - \overline{\alpha} z\right)\left(1 - |z|^2 - 2i\Im\{\alpha\overline{z}\}\right) \\
&= c_{\alpha,z} \left(1 - |z|^2 + 2i\Im\{\overline{\alpha}z\} - \overline{\alpha}z(1-|z|^2) +2i \Im\{\alpha\overline{z}\}\overline{\alpha} z\right) \\
&= c_{\alpha,z} \big[1 - |z|^2 + 2i\Im\{\overline{\alpha}z\} - \left(\Re\{\overline{\alpha} z\} + i\Im\{\overline{\alpha} z\}\right)(1-|z|^2) \\
	&\hspace{1cm}-2i \Im\{\overline{\alpha}z\}(\Re\{\overline{\alpha} z\} + i\Im\{\overline{\alpha} z\})\big] \\
&= c_{\alpha,z} \big[1 - |z|^2 + 2i\Im\{\overline{\alpha}z\} - (1-|z|^2)\Re\{\overline{\alpha} z\} - i(1-|z|^2)\Im\{\overline{\alpha} z\} \\
	&\hspace{1cm}-2i \Im\{\overline{\alpha}z\}\Re\{\overline{\alpha} z\} + 2 (\Im\{\overline{\alpha}z\})^2 \big] 
\end{align*}
Combining real and imaginary terms, then simplifying, gives
\begin{align*}
\tilde{C}_{\varphi_\alpha}(z)&= c_{\alpha,z} \left(1 - |z|^2  - (1-|z|^2)\Re\{\overline{\alpha} z\} + 2 (\Im\{\overline{\alpha}z\})^2 \right)\\
	&\hspace{1cm} + ic_{\alpha,z} \left(2\Im\{\overline{\alpha}z\} - (1-|z|^2)\Im\{\overline{\alpha} z\} -2 \Im\{\overline{\alpha}z\}\Re\{\overline{\alpha} z\} \right) \\
&= c_{\alpha,z} \left((1 - |z|^2)(1-\Re\{\overline{\alpha} z\}) + 2 (\Im\{\overline{\alpha}z\})^2 \right)\\
	&\hspace{1cm}+ ic_{\alpha,z} \Im\{\overline{\alpha}z\} \left(1+|z|^2 - 2\Re\{\overline{\alpha} z\} \right).
\end{align*}
Noting that $c_{\alpha,z} \in \R$ gives the result.
\end{proof}

We can use this information to gather some facts about the geometry of $B(C_{\varphi_\alpha})$. 
\begin{prop}\label{symmetric}
The Berezin range of $C_{\varphi_\alpha}$ on $H^2$ is closed under complex conjugation, and therefore symmetric about the real line. 
\end{prop}
\begin{proof}
Put $z = re^{i\theta}$ and $\alpha = \rho e^{i\psi }$. We claim that $\tilde{C}_{\varphi_\alpha}\left(re^{i\theta}\right) = \overline{\tilde{C}_{\varphi_\alpha}(re^{i(2\psi - \theta)})}$. 
This is the case if and only if $re^{i(2\psi - \theta)} \overline{\varphi_\alpha(re^{i(2\psi - \theta)})} = re^{-i\theta}\varphi_\alpha(re^{i\theta})$ or, equivalently, if and only if $e^{i2\psi} \overline{\varphi_\alpha(re^{i(2\psi - \theta)})} = \varphi_\alpha(re^{i\theta})$.
So let us compute: 
\begin{align*}
e^{2\psi i} \overline{\varphi_\alpha(re^{i(2\psi - \theta)})} &= e^{i2\psi} \ \frac{re^{i(\theta - 2\psi)} - \rho e^{-i\psi}}{1 - \rho e^{i\psi}re^{i(\theta - 2\psi)}} \\
& = \frac{re^{i\theta} - \rho e^{i\psi}}{1 - \rho e^{-i\psi}re^{i\theta}}\\
&= \varphi_\alpha(re^{i\theta}).
\end{align*}
\end{proof}

We point out a corollary of this result that will be important in establishing the characterization of convexity. 

\begin{cor}
If the Berezin range of $C_{\varphi_\alpha}$ on $H^2$ is convex, then $\Re\left\{\tilde{C}_{\varphi_\alpha}(z)\right\} \in B(\tilde{C}_{\varphi_\alpha})$ for each $z \in \D$.
\end{cor}
\begin{proof}
Suppose $B(\tilde{C}_{\varphi_\alpha})$ is convex. Then since $B(C_{\varphi_\alpha})$ is closed under complex conjugation (by Proposition \ref{symmetric}), we have 
\[
\frac12\tilde{C}_{\varphi_\alpha}(z) + \frac12\overline{\tilde{C}_{\varphi_\alpha}(z)} = \Re\left\{\tilde{C}_{\varphi_\alpha}(z)\right\}  \in B(\tilde{C}_{\varphi_\alpha}).
\]
\end{proof}

We these tools in hand, we can provide a characterization of convexity. 
\begin{theorem}\label{Blaschke}
The Berezin range of $C_{\varphi_\alpha}$ on $H^2$ is convex if and only if $\alpha = 0$. 
\end{theorem}
\begin{proof}
If $\alpha = 0$, then $B(C_{\varphi_\alpha}) = \{ 1\}$, which is convex.
Conversely, suppose $B(\tilde{C}_{\varphi_\alpha})$ is convex. Then since $B(C_{\varphi_\alpha})$ is closed under complex conjugation, we have 
\[
\frac12\tilde{C}_{\varphi_\alpha}(z) + \frac12\overline{\tilde{C}_{\varphi_\alpha}(z)} = \Re\left\{\tilde{C}_{\varphi_\alpha}(z)\right\}  \in B(\tilde{C}_{\varphi_\alpha}).
\]
Accordingly, for each $z \in \D$, we can find $w\in \D$ such that 
\[
\tilde{C}_{\varphi_\alpha}(w) = \Re\left\{\tilde{C}_{\varphi_\alpha}(z)\right\}.
\]
In turn, $\Im\left\{ \tilde{C}_{\varphi_\alpha}(w)\right\} = c_{\alpha, w} \Im\{\overline{\alpha}w\} \left(1+|w|^2 - 2\Re\{\overline{\alpha} w\} \right)= 0$, where $c_{\alpha, w}$ is defined as in Lemma \ref{real_imag_parts}.
But because $c_{\alpha, w} > 0$ and $\left(1+|w|^2 - 2\Re(\overline{\alpha} w) \right)>0$ for any $\alpha, w \in \D$, we have $\Im\left\{\tilde{C}_{\varphi_\alpha}(w)\right\} = 0$ if and only if $\Im\{\overline{\alpha}w\} = 0$. 
This says that $\alpha$ and $w$ lie on a line passing through the origin. So we can put $w = r \alpha$ for some $r \in \left(-1/|\alpha|, 1/|\alpha|\right)$.
Now we have 
\begin{align*}
\tilde{C}_{\varphi_\alpha}(w) &= \Re\left\{ \tilde{C}_{\varphi_\alpha}(r\alpha)\right\} \\
&= \frac{(1 - |r\alpha|^2)}{|1 - |r\alpha|^2 + 2i\Im(\alpha\overline{r\alpha})|^2} \left((1 - |r\alpha|^2)(1-\Re\{\overline{\alpha} r\alpha\}) + 2 (\Im\{\overline{\alpha}r\alpha\})^2 \right)\\[.3cm]
&= 1- r|\alpha|^2.
\end{align*}\\
Consequently,  $\left\{\tilde{C}_{\varphi_\alpha}(r\alpha) : r \in \left(-1/|\alpha|, 1/|\alpha|\right) \right\} = (1 - |\alpha|, 1 + |\alpha|)$.
However, putting $z = \rho e^{i\theta}$, an elementary exercise shows that
\[
\lim_{\rho \to 1^-} \tilde{C}_{\varphi_\alpha}(\rho e^{i\theta}) = 
\begin{cases} 
      0, & \alpha \neq 0 \\
      1, & \alpha = 0
\end{cases}
\]
This tells us that when $\alpha \neq 0$, given $\epsilon$ with $0 < \epsilon < 1 - |\alpha|$, there exists a point $z$ such that $\left|\Re\left\{\tilde{C}_{\varphi_\alpha}(z)\right\}\right| < \epsilon$. But if $\tilde{C}_{\varphi_\alpha}(w) = \Re\left\{\tilde{C}_{\varphi_\alpha}(z)\right\}$, this is a contradiction since $\tilde{C}_{\varphi_\alpha}(w) \in (1 - |\alpha|, 1 + |\alpha|)$. Thus, $B(\tilde{C}_{\varphi_\alpha})$ cannot be convex unless $\alpha = 0$. 
\end{proof}

\section{Other Examples and Further Directions}\label{questions}
We conclude with a few questions and remarks, motivated by some examples. 
We start with questions that naturally follow the results regarding composition operators on $H^2$ in the previous section.  
These operators have symbols which are automorphisms of the disk, taking the general form
\[
b(z) = \zeta \frac{z - \alpha}{1 - \overline{\alpha}z},
\]
where $\zeta \in \T$ and $\alpha \in \D$. 
\begin{question}
In $H^2$, can one characterize the convexity of the Berezin range for the composition operator with symbol $b$, with $b$ defined as above, by combining the results in Theorems \ref{elliptic} and \ref{Blaschke}? 
\end{question}

Even more generally, the examples in Section \ref{comp_ops} belong to the class of composition operators having symbols known as M\"obius transformations 
\[
M(z) = \frac{az+b}{cz+d},
\]
where $a,b,c,d \in \C$ and $ad-bc \neq 0$.
Another natural step would be to consider the class of operators $C_M$ acting on $H^2$. 
\begin{question}
Given the composition operator $C_M$ acting on $H^2$, with $M$ defined as above, what are necessary and sufficient conditions for $B(C_M)$ to be convex? 
\end{question}

\begin{figure}[H]\label{moebius}
\includegraphics[width=.82\textwidth]{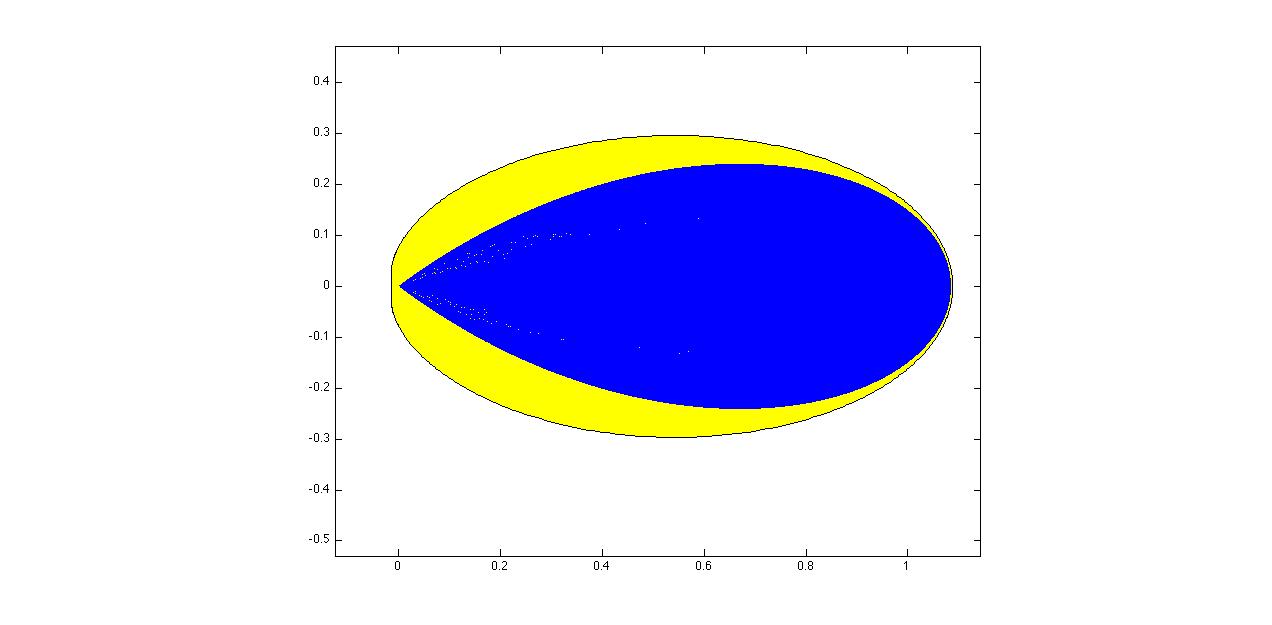}
\caption{$W(C_M)$ and $B(C_M)$ on $H^2$ for $M(z) = \frac{4 + 2z}{9 - z}$, which appears to be convex.}
\end{figure}
The Berezin range in the above figure seems to be symmetric about the real line, as was also seen in Figure 3 and in the first plot in Figure 2, but not in the second plot of Figure 2. This begs a question:
\begin{question}
For a composition operator $C_\varphi$ acting on $H^2$, when is $B(C_\varphi)$ symmetric about the real line? 
\end{question}
Recall that symmetry described in Proposition \ref{symmetric} was a key part in proving Theorem \ref{Blaschke}. 
It has also been of interest to understand when the numerical range of an operator can be a disk or an ellipse (cf. the Elliptical Range Theorem), and to deduce other properties concerning the circular symmetry of the numerical range (see \cite{MR2826148, MR3005640, MR1971096}). 
Similar questions can be asked of the Berezin range: 
\begin{question}
Given a bounded operator $T$ on $\h$, when is $B(T)$ an ellipse or circular disk? 
\end{question}
In fact, Karaev showed the Berezin range of multiplication by $z^n$ acting on a certain model space is a disk \cite[Example~2.1(a)]{MR3079840}.
The following example points toward the possibility that the Berezin range can be a \textit{circular disk} for other operators. 
\begin{figure}[H]
\includegraphics[width=.82\textwidth]{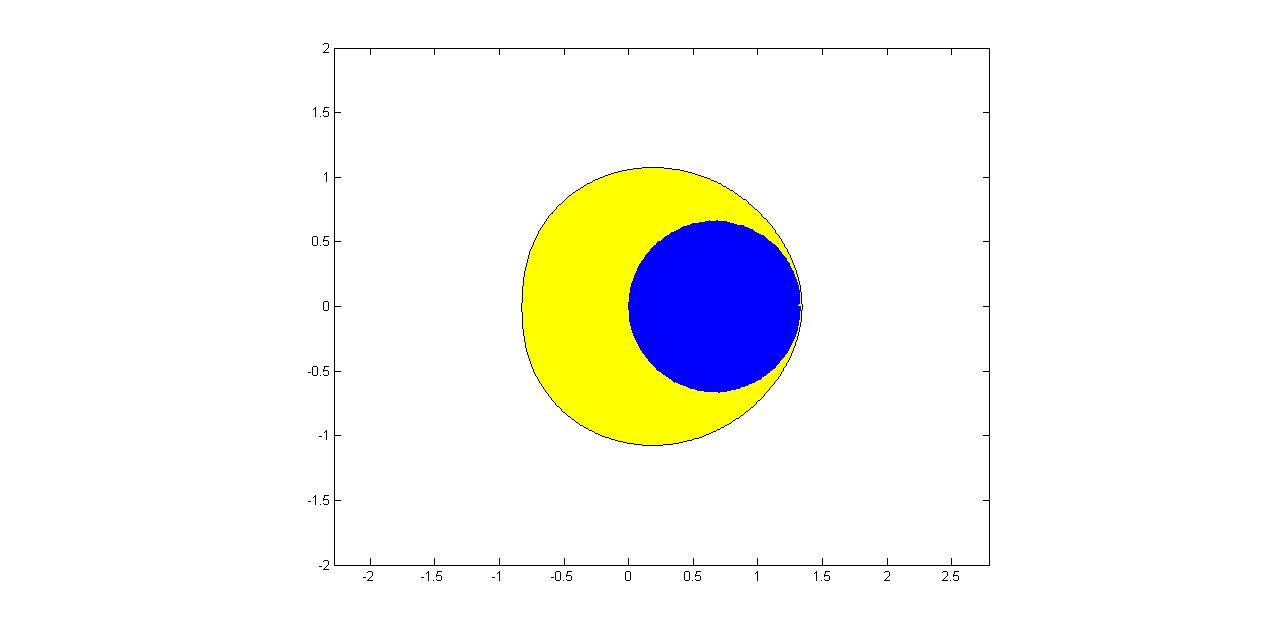}
\caption{$W(C_\varphi)$ and $B(C_\varphi)$ on $H^2$ for $\varphi(z) = \frac{1 + z}{2}$.}
\end{figure}

In light of the Axler-Zheng Theorem and its descendants, pointed out in Section \ref{background}, the Berezin transform is arguably most useful in the Bergman space, and, admittedly, the Berezin range may be more natural to consider in this setting. 
\begin{question}
Given a class of concrete operators acting on the Bergman space, what can be said about the convexity of the Berezin range of these operators?  
\end{question}
Of the composition operator results presented here, the obstruction to providing immediate analogous results on the Bergman space is the increased complexity of the reproducing kernel, which is given by $k_w(z) = \frac{1}{(1 - \overline{w}z)^2}$. In general, one may replace the Bergman space in the question above with any RKHS of holomorphic functions.

Looking back to Theorem \ref{mult_op_con}, the multiplication operators considered were uncomplicated. In general, if $\h$ is a closed subspace of a Banach space $Y$, one may take a function $g \in \operatorname{Mult}(Y)$ and consider the generalized \textit{Toeplitz operator} $T_g$ on $\h$, given by 
\[
T_g f = P_\h gf,
\]
where $P_\h$ is the orthogonal projection from $Y$ onto $\h$.

\begin{question}
Given a (generalized) Toeplitz operator $T_g$ acting on an RKHS $\h$, what can be said about the convexity of $B(T_g)$?  
\end{question}
We point here to the case where $g(z) = z$ and $\h$ is a model space (see \cite{MR3309355} for background on model spaces). The Toeplitz operator in this case is known as a \textit{compression of the shift} and the numerical range of this operator has been studied in both one and two variables \cite{MR3764150, MR4061945}. In general, these types of operators are known as truncated Toeplitz operators, and it has been shown that the compactness of these operators can be characterized in terms of the vanishing boundary behavior of the Berezin transform \cite{MR2766252}.

On many RKHSs, it is well known that various properties (e.g. boundedness or compactness) of certain operators can be deduced from considering only the action of the operator on the set of normalized reproducing kernels. Results of this type are known as \textit{reproducing kernel theses}. The literature surrounding this idea is extensive, so we point to \cite{MR3250359} for introduction and further reading. Many of the results mentioned in Section \ref{background} can be likened to reproducing kernel theses. One may ask if such results exist between the Berezin and numerical ranges. 

\begin{question}
Given an operator $T$ on an RKHS $\h$, are there any properties of $W(T)$ that can deduced from $B(T)$? 
\end{question}
For example, can one relate the Berezin radius and the numerical radius of an operator? An elementary estimate gives $b(T) \le w(T)$. However, one might ask for a sharp constant $C$ (depending on $T$) so that $w(T) \le C b(T)$. In many of the examples we have presented, it seems that the quantities are equal. Can one characterize when this is the case? The upshot in proving an equality would be that the radius of the Berezin range is much easier to compute than the numerical radius. In the case of Toeplitz operators acting on $H^2$, and for some truncated Toeplitz operators, it is known that these quantities are equal \cite{MR3079840}. However, these quantities are not equal in general \cite[Example~2]{MR3079840}.  
On $H^2$, it is known that the numerical radius of an operator can be bounded above and below by the Berezin radius of certain conjugates of the operator \cite[Proposition~1]{MR3522132}. In this vein, we end by mentioning recent interest in establishing inequalities for the Berezin radius (e.g. see \cite{MR3535209, MR3744172, MR3723519, MR3938330}).

\subsection*{Acknowledgments}
The second author sends abundant thanks to Kelly Bickel, John McCarthy, and Jeff Norton for helpful discussion.

\bibliographystyle{abbrv}
\bibliography{masterbib}

\end{document}